\input{aipcheck}

\documentclass[final]{aipproc}
\layoutstyle{8x11single}

\usepackage{autograph,epic}
\usepackage{latexsym,bezier,amsbsy,psfrag,color,enumerate,amsfonts,amsmath,amscd,amssymb,amsthm}

\setloopdiam{17}\setprofcurve{10}

\newtheorem{definition}{Definition}
\newtheorem{theorem}[definition]{Theorem}
\newtheorem{lemma}[definition]{Lemma}
\newtheorem{corollary}[definition]{Corollary}

\begin{document}

\title{Isomorphism classification of infinite Sierpi\'{n}ski carpet graphs}

\classification{}
\keywords{Sierpi\'{n}ski carpet, graph isomorphism, cofinality, rooted graph.\\ \textbf{Mathematics Subject Classification (2010)}: 05C10, 05C60, 05C63, 05A16.}

\author{Daniele D'Angeli}{
address={Institut f\"{u}r mathematische Strukturtheorie (Math C), Technische Universit\"{a}t Graz, Steyrergasse 30, 8010 Graz, Austria.\footnote{The first author was supported by Austrian Science Fund project FWF P24028-N18.} \\ Email address: dangeli@math.tugraz.at}}

\author{Alfredo Donno}{
address={Universit\`{a} degli Studi Niccol\`{o} Cusano - Via Don Carlo Gnocchi, 3 00166 Roma, Italia. \\ Email address: alfredo.donno@unicusano.it}}

\begin{abstract}
For each infinite word over a given finite alphabet, we define an increasing sequence of rooted finite graphs,
that can be thought as approximations of the famous Sierpi\'{n}ski
carpet. These sequences naturally converge to an infinite rooted limit graph. We show that there are uncountably many classes of isomorphism of such limit
graphs, regarded as unrooted graphs.
\end{abstract}

\maketitle


\section{Introduction}
The famous Sierpi\'{n}ski carpet is a fractal introduced by
W. Sierpi\'{n}ski in 1916 \cite{courbe}, which can be considered as
a two-dimensional generalization of the Cantor set. Like the
well-known Sierpi\'{n}ski gasket, the carpet has a self-similar
structure: the former is a finitely ramified fractal (in Physics, a fractal is said to be finitely ramified if it can be disconnected by removing a finite number of points),
whereas the latter is an infinitely ramified fractal. Both the structures have been widely studied in the literature,
from several points of view. In particular, the study of critical
phenomena and physical models - the Ising model \cite{bonnier, ising, tutte2, gefen2, gefen3, solvable}, the dimer model
\cite{chang, dimeri}, the percolation model \cite{shinoda2,
percolation} - on the Sierpi\'{n}ski carpet and on the
Sierpi\'{n}ski gasket has been the focus of several works in the
last decades. We also want to mention the paper \cite{brownian},
where Brownian motion on some fractal sets generalizing the
Sierpi\'{n}ski carpet is studied. The Sierpi\'{n}ski carpet and
other fractal sets have also interesting applications in Telecommunications engineering as in the
construction of antennas and in the study of propagation phenomena (e.g.,
\cite{alessio}).

In this short paper, we associate with any infinite word, over a
finite alphabet, an increasing sequence of finite rooted graphs, which
represent a finite discrete approximation of the classical Sierpi\'{n}ski
carpet. Each of our sequences naturally converges to an infinite
rooted limit graph. We study the isomorphism properties of these
limit graphs, regarded as unrooted graphs, proving that there
exist uncountably many classes of isomorphism.

This work follows the paper \cite{io}, where the same problem is studied for a sequence of graphs
inspired by the Sierpi\'{n}ski gasket. An analogous
problem is studied in \cite{JMD}, where the sequence of Schreier
graphs associated with the action of the self-similar
Basilica group on the rooted binary tree is considered.

\section{Carpet graphs}

Let us start by fixing two finite alphabets $X=\{0,1,\ldots, 7\}$ and $Y=\{a,b,c,d\}$, and let $X^\infty=\{x_1x_2\ldots : x_i\in X\}$ be the set of all infinite words over the alphabet $X$. Let $C_4$ be the cyclic graph of length $4$ whose vertices will be denoted by $a, b, c, d$. We choose an embedding of this graph into the plane in such a way that all sides are parallel to the coordinate axes and $a$ is the left vertex of the bottom edge, and $b,c,d$ correspond to the other vertices by following the anticlockwise order (Fig. \ref{fig1}). \\

\textit{\textbf{ Construction of the graph $\Gamma_w$}}. Take an
infinite word $w=yx_1x_2\ldots \in Y\times X^\infty$. We denote by
$w_n$ the prefix $yx_1\ldots x_{n-1}$ of length $n$ of $w$.

\begin{definition}
The infinite carpet graph $\Gamma_{w}$ is the rooted
graph inductively constructed as follows.
\begin{itemize}
  \item \textbf{Step $1$.} The graph $\Gamma_w^1$ is the cyclic graph $C_4$ rooted at the vertex $y$.
  \item \textbf{Step $n\to n+1$.} Take $8$ copies of $\Gamma_{w}^n$ and glue them
  together on the model graph $\overline{\Gamma}$, in such a way that these copies occupy the positions indexed by $0,1,\ldots, 7$ in $\overline{\Gamma}$ (Fig. \ref{fig1}). Note that each copy shares at most one (extremal) side with
  any other copy. As a root for the new rooted graph $\Gamma_w^{n+1}$, we choose the root of the copy of $\Gamma_w^n$ occupying the position indexed by the letter $x_n$. We identify the root of $\Gamma_w^{n+1}$ with the finite word $w_{n+1}=yx_1\ldots x_n$.
  \item \textbf{Limit.} $\Gamma_{w}$ is the infinite rooted graph obtained as the limit of the sequence of finite rooted graphs
  $\{\Gamma_{w}^n\}_{n\geq 1}$, whose root is naturally identified with the infinite word $w$.
\end{itemize}
\end{definition}
\vspace{0.5cm}

The limit in the previous definition means that, for each $r>0$, there exists $n_0\in \mathbb{N}$ such that the ball $B_{\Gamma_w}(w,r)$ of radius $r$ rooted at $w$ in $\Gamma_w$ is isomorphic to the ball $B_{\Gamma_w^n}(w_n,r)$ of radius $r$ rooted at $w_n$ in $\Gamma_{w}^n$, for every $n\geq n_0$ (Gromov-Hausdorff topology).

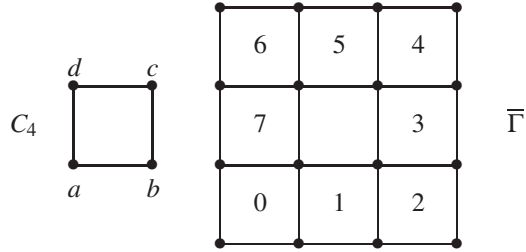
\begin{figure}[h]
\begin{picture}(230,120)\unitlength=0,15mm
\letvertex a=(30,70)\letvertex b=(100,70)\letvertex c=(100,140)\letvertex d=(30,140)

\letvertex e=(160,0)\letvertex f=(230,0)\letvertex g=(300,0) \letvertex h=(370,0)\letvertex i=(370,70)
\letvertex l=(370,140)\letvertex o=(230,210)\letvertex n=(300,210) \letvertex m=(370,210)
\letvertex p=(160,210)\letvertex q=(160,140)\letvertex r=(160,70)
\letvertex s=(230,70)\letvertex v=(230,140)\letvertex t=(300,70)\letvertex u=(300,140)
 \put(-25,100){$C_4$}   \put(415,100){$\overline{\Gamma}$}

\put(25,42){$a$} \put(25,149){$d$}\put(95,42){$b$}\put(95,149){$c$}

\drawvertex(a){$\bullet$}
\drawvertex(b){$\bullet$}\drawvertex(c){$\bullet$}\drawvertex(d){$\bullet$}

\drawvertex(e){$\bullet$}
\drawvertex(f){$\bullet$}\drawvertex(g){$\bullet$}\drawvertex(h){$\bullet$}\drawvertex(i){$\bullet$}
\drawvertex(l){$\bullet$}\drawvertex(m){$\bullet$}\drawvertex(n){$\bullet$}\drawvertex(o){$\bullet$}
\drawvertex(p){$\bullet$}\drawvertex(q){$\bullet$}\drawvertex(r){$\bullet$}\drawvertex(s){$\bullet$}
\drawvertex(t){$\bullet$}\drawvertex(u){$\bullet$}\drawvertex(v){$\bullet$}
\drawundirectededge(b,a){} \drawundirectededge(c,b){}
\drawundirectededge(a,d){} \drawundirectededge(c,d){}

\drawundirectededge(f,e){} \drawundirectededge(g,f){}
\drawundirectededge(h,g){} \drawundirectededge(i,h){}
\drawundirectededge(i,l){} \drawundirectededge(l,m){}
\drawundirectededge(n,o){} \drawundirectededge(m,n){}
\drawundirectededge(p,q){} \drawundirectededge(p,o){}
\drawundirectededge(q,r){} \drawundirectededge(r,e){}

\drawundirectededge(r,s){} \drawundirectededge(s,t){}\drawundirectededge(t,i){}
\drawundirectededge(f,s){}\drawundirectededge(s,v){} \drawundirectededge(v,o){}
\drawundirectededge(q,v){}\drawundirectededge(u,v){} \drawundirectededge(u,l){}
\drawundirectededge(g,t){}\drawundirectededge(u,t){} \drawundirectededge(u,n){}

\put(190,30){$0$}\put(190,100){$7$}\put(190,170){$6$}
\put(330,30){$2$}\put(330,100){$3$}\put(330,170){$4$}
\put(260,30){$1$}\put(260,170){$5$}
\end{picture}
\caption{The cyclic graph $C_4$ and the model graph $\overline{\Gamma}$.}\label{fig1}
\end{figure}
 \vspace{0.5cm}

Observe that, for all $v,w\in Y\times X^\infty$, the graph $\Gamma_v^n$ is isomorphic to $\Gamma_w^n$ as an unrooted graph. When we will refer to this unrooted graph, we will use the notation $\Gamma_n$. One can check that the number of vertices of $\Gamma_n$ is
$$
\frac{11}{70}8^n+\frac{8}{15}3^n+\frac{8}{7}.
$$

Moreover, note that two distinct finite words $v_n$ and $w_n$ may correspond to the same vertex of $\Gamma_n$. In order to detect such points, we need to define the set $B_n$ of the \textit{boundary vertices} of the graph $\Gamma_n$, for every $n\geq 2$: they are exactly all the vertices belonging to the four extremal sides of $\Gamma_n$. Note that the boundary vertices have degree less or equal to $3$, and they correspond to words of the following form:
\begin{itemize}
\item $w_n=yx_1\ldots x_{n-1}$, with $y\in \{a,b\}$, $x_i\in \{0,1,2\}$ (boundary vertices of the \textit{bottom} side);
\item $w_n=yx_1\ldots x_{n-1}$, with $y\in \{b,c\}$, $x_i\in \{2,3,4\}$ (boundary vertices of the \textit{right} side);
\item $w_n=yx_1\ldots x_{n-1}$, with $y\in \{c,d\}$, $x_i\in \{4,5,6\}$ (boundary vertices of the \textit{top} side);
\item $w_n=yx_1\ldots x_{n-1}$, with $y\in \{d,a\}$, $x_i\in \{6,7,0\}$ (boundary vertices of the \textit{left} side).
\end{itemize}

In particular, the vertices of degree $2$ are indexed by the words $a0^{n-1}, b2^{n-1}, c4^{n-1}, d6^{n-1}$. One can check that each extremal boundary side of $\Gamma_n$ contains $3^{n-1}+1$ vertices, so that the number of the boundary vertices of $\Gamma_n$ is $4\cdot 3^{n-1}$.

Boundary vertices are those in $\Gamma_n$ that can be possibly identified in $\Gamma_{n+1}$. For instance, it is easy to check that, if $w=a206\ldots$ and $v=d467\ldots$, then $w_3$ (resp. $v_3$) corresponds to a vertex on the bottom (resp. top) boundary side of $\Gamma_3$, but the words $w_4$ and $v_4$ correspond to the same vertex of $\Gamma_4$. Analogous considerations can be done for other boundary vertices. 

\begin{definition}
Two infinite words $v,w\in Y\times X^\infty$ are said to be \textit{cofinal} if they differ only for a finite number of letters.
\end{definition}

Cofinality is clearly an equivalence relation, that we will denote by $\sim$. Given $v,w\in Y\times X^\infty$, if there exists $n_0\in \mathbb{N}$ such that $v_n$ and $w_n$ correspond to the same vertex of the finite graph $\Gamma_n$ for every $n\geq n_0$, then $v=v_{n_0}u$ and $w=w_{n_0}u$, for some $u\in X^{\infty}$. On the other hand, it is not difficult to check that all the vertices belonging to the same infinite graph $\Gamma_w$, with $w\in Y\times X^\infty$, are cofinal with $w$. In what follows, we consider just one representative word for vertices that can be identified with different words.

For each $n\geq 2$, we inductively define the \textit{internal boundary} $I_b(n)$ of $\Gamma_n$ as follows. For $n=2$, $I_b(2)$ consists of the $4$ vertices of the middle square of $\Gamma_2$. For $n=3$, $I_b(3)$ consists of the $8\cdot 4$ vertices of $\Gamma_3$ corresponding to the images of the vertices of $I_b(2)$ in each copy of $\Gamma_2$ in the graph $\Gamma_3$, together with the $12$ vertices of the middle square of $\Gamma_3$. By recursion, $I_b(n)$ will consist of the vertices of $\Gamma_n$ coming from vertices of $I_b(n-1)$, together with the vertices of the \lq\lq new generation\rq\rq middle square of $\Gamma_n$.\\ \indent For every infinite word $w\in Y\times X^\infty$, we define the sequence $\{d_i^w\}_{i\geq 2}$ as
$$
d_i^w=\min \{d(w_i, v) \ : \ v\in I_b(i)\}.
$$
Here, $d$ denotes the usual geodesic graph distance.

\begin{lemma}\label{lemma}
Let $w,v\in Y\times X^\infty$.
\begin{enumerate}[1)]
\item $d_i^w=d_i^v$ for each $2\leq i\leq N$ if and only if there exists an isomorphism between $\Gamma^N_w$ and $\Gamma^N_v$ mapping $w_N$ to $v_N$.
\item Suppose that there exists $k\geq 2$ such that $d_k^w\neq d_k^v$. Then there exists $r>0$ such that the balls $B_{\Gamma_w}(w,r)$ and $B_{\Gamma_v}(v,r)$ are not isomorphic.
    \end{enumerate}
\end{lemma}

\begin{proof}
\begin{enumerate}[1)]
\item It follows from an easy geometric argument.
\item Suppose, by contradiction, that there exists an isomorphism $\phi$ between $\Gamma_w$ and $\Gamma_v$ such that $\phi(w)=v$. Then, the restriction of $\phi$ to the subgraph of $\Gamma_w$ rooted at $w$, and isomorphic to the graph $\Gamma^k_w$ rooted at $w_k$, is also an isomorphism of rooted graphs. The absurd follows from 1).
    \end{enumerate}
\end{proof}
From Lemma \ref{lemma} we obtain the following classification result. We are now considering the graphs $\{\Gamma_w:  w\in Y\times X^\infty\}$ as unrooted graphs and we ask whether two such graphs are isomorphic.
\begin{theorem}
Let $v=yx_1x_2\ldots,w=y'x_1'x_2'\ldots\in Y\times X^\infty$ and let $G$ be the subgroup of $Sym(X)$ generated by the permutations $\{(1357)(2460),(04)(13)(57)\}$, isomorphic to the dihedral group of order $8$. Then
$\Gamma_v\simeq\Gamma_w$ if and only if there exists $\sigma\in G$ such that
$$
x_1'x_2'\ldots \sim \sigma(x_1x_2\ldots) :=\sigma(x_1) \sigma(x_2)\ldots.
$$
\end{theorem}
\begin{proof}
Let $\phi:\Gamma_v\rightarrow\Gamma_w$ be an isomorphism and suppose that $\phi(v)=u\in \Gamma_w$, with $B_{\Gamma_v}(v,r)\simeq B_{\Gamma_w}(u,r)$ for every $r>0$. Let $u= \overline{y}\,\overline{x}_1\overline{x}_2\ldots$ and suppose that $\sigma(x_1x_2\ldots)$ is not cofinal to $\overline{x}_1\overline{x}_2\ldots$, for every $\sigma\in G$. Observe that $G$ is the group of isometries of $\Gamma_w^n$, for each $n$. Then, by construction of the $\Gamma_w$'s and an easy geometric argument, we deduce that there exist infinitely many positive integers $\{i_k\}_{k\geq 1}$ such that $d_{i_k}^u\neq d_{i_k}^v$. Lemma \ref{lemma} implies that there exists $r'>0$ such that $B_{\Gamma_v}(v,r')\not\simeq B_{\Gamma_w}(u,r')$. A contradiction.

Vice versa, suppose that $\sigma(x_1x_2\ldots)=\overline{x}_1\overline{x}_2\ldots$ for some $\sigma\in G$. Put $\phi(v)=\overline{y}\sigma(x_1x_2\ldots)=u$. Since $G$ is the group of isometries of $\Gamma_n$, the distances $d_n^v$ and $d^u_n$ coincide for every $n$. This ensures that $\phi$ is an isomorphism.
\end{proof}

\begin{corollary}
There exist uncountably many isomorphism classes of graphs $\Gamma_w$, $w\in Y\times X^\infty$, regarded as unrooted graphs.
\end{corollary}

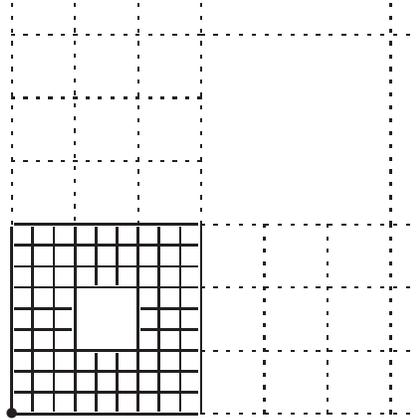
\begin{figure}[h]
\begin{picture}(230,190)\unitlength=0,14mm   \thinlines{
\letvertex a=(40,40)\letvertex b=(60,40)\letvertex c=(80,40)\letvertex d=(100,40)
\letvertex e=(120,40)\letvertex f=(140,40)\letvertex g=(160,40) \letvertex h=(180,40)\letvertex i=(200,40)
\letvertex l=(220,40)

\letvertex A=(40,220)\letvertex B=(60,220)\letvertex C=(80,220)\letvertex D=(100,220)
\letvertex E=(120,220)\letvertex F=(140,220)\letvertex G=(160,220) \letvertex H=(180,220)\letvertex I=(200,220)
\letvertex L=(220,220)

\letvertex AA=(40,200)\letvertex BB=(40,180)\letvertex CC=(40,160)\letvertex DD=(40,140)
\letvertex EE=(40,120)\letvertex FF=(40,100)\letvertex GG=(40,80) \letvertex HH=(40,60)

\letvertex aa=(220,200)\letvertex bb=(220,180)\letvertex cc=(220,160)\letvertex dd=(220,140)
\letvertex ee=(220,120)\letvertex ff=(220,100)\letvertex gg=(220,80) \letvertex hh=(220,60)

\letvertex Ee=(120,160)\letvertex Ff=(140,160) \letvertex eE=(120,100)
\letvertex fF=(140,100)\letvertex M=(100,140)\letvertex N=(100,120)
\letvertex O=(160,140)\letvertex P=(160,120)

\drawvertex(a){$\bullet$}

\drawundirectededge(a,A){} \drawundirectededge(b,B){}
\drawundirectededge(c,C){} \drawundirectededge(d,D){}
\drawundirectededge(G,g){} \drawundirectededge(h,H){}
\drawundirectededge(i,I){} \drawundirectededge(l,L){}

\drawundirectededge(A,L){} \drawundirectededge(AA,aa){}
\drawundirectededge(BB,bb){} \drawundirectededge(CC,cc){}
\drawundirectededge(FF,ff){} \drawundirectededge(GG,gg){}
\drawundirectededge(HH,hh){} \drawundirectededge(a,l){}

\drawundirectededge(DD,M){} \drawundirectededge(O,dd){}\drawundirectededge(EE,N){}
\drawundirectededge(P,ee){}

\drawundirectededge(E,Ee){}
\drawundirectededge(F,Ff){}\drawundirectededge(e,eE){} \drawundirectededge(f,fF){}     }


\letvertex a=(40,220)\letvertex d=(100,220)
\letvertex g=(160,220) \letvertex l=(220,220)

\letvertex A=(40,400)\letvertex D=(100,400)\letvertex G=(160,400) \letvertex L=(220,220)

\letvertex CC=(40,340)\letvertex FF=(40,280)

\letvertex cc=(220,340)\letvertex ff=(220,280)

\dashline[0]{4}(100,430)(100,200) \dashline[0]{4}(40,430)(40,220)\dashline[0]{4}(160,430)(160,220)\dashline[0]{4}(220,220)(220,430)
\dashline[0]{4}(40,400)(220,400)\dashline[0]{4}(40,340)(220,340)\dashline[0]{4}(220,280)(40,280)

\dashline[0]{4}(280,220)(280,40) \dashline[0]{4}(220,220)(220,40)\dashline[0]{4}(340,220)(340,40)\dashline[0]{4}(400,40)(400,220)
\dashline[0]{4}(220,220)(430,220)\dashline[0]{4}(220,160)(430,160)\dashline[0]{4}(430,100)(220,100) \dashline[0]{4}(430,40)(220,40)

\dashline[0]{4}(220,400)(430,400) \dashline[0]{4}(400,430)(400,220)

\end{picture}
\caption{The rooted graph $\Gamma_{a0^\infty}$.}
\end{figure}


\bibliographystyle{aipproc}   

 \end{document}